\documentclass[a4paper,12pt]{amsart}

\usepackage{setspace,graphicx}     
\usepackage{stackrel}                                          
\usepackage{amssymb}
\usepackage{verbatim} 

\usepackage{blindtext}
\usepackage{tabularx}
\usepackage{longtable}

\usepackage{amsmath} 
\usepackage{amsthm} 
\usepackage{mathrsfs} 
\usepackage{amsfonts} 
\usepackage{latexsym} 
\usepackage{amscd} 
\usepackage[latin1]{inputenc} 
\usepackage[english]{babel} 
\usepackage{enumerate} 
\usepackage{afterpage}
\usepackage{tikz-cd}
 
\allowdisplaybreaks

\def\p{\mathfrak{p}}

\newcommand{\F}{\mathbb{F}}  
\newcommand{\N}{\mathbb{N}} 
\newcommand{\Z}{\mathbb{Z}} 
\newcommand{\Q}{\mathbb{Q}}  
\newcommand{\D}{\mathbf{D}} 
 
\newcommand{\A}{\mathscr{A}}
\newcommand{\B}{\mathscr{B}}
\newcommand{\NN}{\mathscr{N}}
\newcommand{\MM}{\mathscr{M}}
\newcommand{\QQ}{\mathscr{Q}}
\newcommand{\LL}{\mathscr{L}}
\newcommand{\BB}{\mathcal{B}}

\newtheorem{Theorem}{Theorem}[]
\newtheorem{Lemma}[Theorem]{Lemma} 
 
\newtheorem{Proposition}[Theorem]{Proposition}

\theoremstyle{definition}
\newtheorem{Definition}[Theorem]{Definition}

\newtheorem{Remark}[Theorem]{Remark} 
\newtheorem{Example}[Theorem]{Example}

\DeclareMathOperator{\lcm}{lcm}
\DeclareMathOperator{\Spec}{Spec}

\newtheorem{ind}[]{{\rm\it Indice}}

\begin{document}

\author{Seoyoung Kim}
\email{Seoyoung\_Kim@math.brown.edu}
\address{Mathematics Department, Brown University, Box 1917, 151 Thayer Street, Providence, RI 02912 USA}
\title[Terms having a fixed G.C.D. with their index in an EDS]{The density of the terms in an elliptic divisibility sequence having a fixed G.C.D. with their index}

\date{\today}

\begin{abstract}
Let $\D=(D_{n})_{n\geq 1}$ be an elliptic divisibility sequence associated to the pair $(E,P)$. For a fixed integer $k$, we define $\A_{E,k}=\{n\geq 1 : \gcd(n,D_{n})=k\}$. We give an explicit structural description of $\A_{E,k}$. Also, we explain when $\A_{E,k}$ has positive asymptotic density using bounds related to the distribution of trace of Frobenius of $E$. Furthermore, we get explicit density of $\A_{E,k}$ using the M\"obius function.
\end{abstract}

\maketitle

%
%

\section{Introduction}
\noindent

Let $E/\Q$ be an elliptic curve defined by a Weierstrass equation and choose a nontorsion point $P\in E(\Q)$.

\begin{Definition}
The {\it elliptic divisibility sequence} (EDS) associated to the pair $(E,P)$ is the sequence $\D=(D_{n})_{n\geq 1}:\N \rightarrow \N$ defined by writing the $x$-coordinate of $[n]P$ as a fraction in lowest terms
$$x([n]P)=\frac{A_{n}}{D_{n}^{2}}\in\Q,$$
where $[n]$ is the multiplication-by-$n$ map. The EDS is {\it minimal} if $E/\Q$ is defined by a minimal Weierstrass equation and we call it {\it normalized} if $D_{1}=1$. It is worth noting that every EDS could be normalized by change of variables in the original Weierstrass equation with which one can get the new EDS $(D_{n}/D_{1})_{n \geq 1}$. Also, note that not every EDS can be both normalized and minimal.
\end{Definition}

As one can guess from its name, EDS is a divisibility sequence, that is to say,
$$m\mid n~\Longrightarrow~D_{m}\mid D_{n}$$
which can be deduced from \cite[Lemma 5]{SS}. Moreover, it satisfies the strong divisibility condition
$$\gcd(D_{n},D_{m})=D_{\gcd(n,m)}.$$
From now on, we assume that $(D_{n})_{n\geq 1}$ is minimal.
Note that for all good primes $p$, we know
$$p \mid D_{n} ~\Leftrightarrow ~ [n]P\equiv O ~(mod~p).$$
For more general statement, one can refer to Lemma \ref{PreLem}. Also, we can be naturally interested in the following values.

\begin{Definition}
Define $r_{n}=r_{n}(\D)$ the {\it rank of apparition} (or {\it entry point}) of $n$ in $\D$ to be
$$r_{n}=\min\{r\geq 1 : n\mid D_{r}\}.$$
That is to say, $r_{n}$ is the minimal index of the terms of $\D$ which are divisible by $n$. Also define
$$g(n)=\gcd(n,D_{n})$$
and
$$l(n)=\lcm(n,r_{n}).$$
\end{Definition}

For each positive integer $k$, we are interested in the following set
$$\A_{E,k}=\{n\geq 1 : g(n)=k\}.$$

Moreover, for a given set $S$ with positive integers, define the {\it asymptotic density} $\textbf{d}(S)$ as follows:
$$\textbf{d}(S)=\lim_{x\to\infty}\frac{\#(S\cap[1,x])}{x}$$

We are interested in the asymptotic density of $\A_{E,k}$, i.e., we mostly let $S=\A_{E,k}$ for some elliptic curve $E$ and a fixed integer $k$. Sanna and Tron \cite{ST} considered the above set with the Fibonacci sequence. More precisely, if we denote $(F_{n})_{n\geq 1}$ the Fibonacci sequence, for each positive integer $k$, the set 
$$\A_{F,k}=\{n\geq 1:\gcd(n,F_{n})=k\}$$
satisfies the following properties:

\begin{Theorem}{(Sanna, Tron \cite{ST})}
Let $k$ be a positive integer
\begin{enumerate}[(a)]
\item The asymptotic density of $\A_{F,k}$ exists.
\item The following are equivalent:
\begin{enumerate}[(i)]
\item $\textbf{d}(\A_{F,k})>0.$
\item $\A_{F,k}\neq \emptyset.$
\item $k=\gcd(l(k),F_{l(k)}).$
\end{enumerate}
\end{enumerate}
\end{Theorem}

\begin{Theorem}{(Sanna, Tron \cite{ST})}
For each positive integer $k$, we have
$$\textbf{d}(\A_{F,k})=\sum^{\infty}_{d=1}\frac{\mu(d)}{l(dk)},$$
where $\mu$ is the M\"obius function.
\end{Theorem}

We prove analogous results for elliptic divisibility sequences. However, elliptic divisibility sequences have more diverse structure. Especially, different elliptic curves with different fixed points can result various elliptic divisibility sequences. The following are the main results of this note.

\begin{Theorem}
For each positive integer $k$ with $\A_{E,k}\neq \emptyset$, denote
$$\LL_{k}=\left\{ p : p|k \right\}\cup \left\{ \frac{l(kp)}{l(k)}:p\nmid k\right\}.$$
Define
$$\NN(\LL_{k})=\{n\geq 1 : s\nmid n~\text{for all}~s\in\LL_{k} \}.$$
Then
$$\A_{E,k}=\{l(k)m : m\in \NN(\LL_{k})\}.$$
\end{Theorem}

\begin{Theorem}
Assume $E/\Q$ to be an elliptic curve which satisfies one of the following:
\begin{enumerate}
\item $E/\Q$ does not have complex multiplication.
\item $E/\Q$ has complex multiplication, and assume either the generalized Riemann Hypothesis or that $E/\Q$ is finitely anomalous, that is to say, for all but finitely many prime $p$, the reduction of $E/\Q$ satisfies
$$E(\mathbb{F}_{p})\neq p.$$
\end{enumerate}
Then for every integer $k>0$:
\begin{enumerate}[(a)]
\item The asymptotic density $\textbf{d}(\A_{E,k})$ exists.
\item $\textbf{d}(\A_{E,k})>0$ if and only if $\A_{E,k}\neq \emptyset$.
\end{enumerate}
\end{Theorem}

Finitely anomalous elliptic curves do exist, e.g., elliptic curves with a  nontrivial torsion point in $E(\Q)$. One can find more instances in Section 3.

\begin{Theorem}
Let $E/\Q$ be an elliptic curve which satisfies one of the following:
\begin{enumerate}
\item $E/\Q$ does not have complex multiplication.
\item $E/\Q$ has complex multiplication, and assume either the generalized Riemann Hypothesis or $E/\Q$ is finitely anomalous.
\end{enumerate}
Let $k$ be a fixed positive integer. Define $\B_{E,k}$ to be the set of positive integers $n$ satisfying the following conditions.
\begin{enumerate}[(a)]
	\item $k \mid g(n).$
	\item If $p\mid g(n)$ for some prime $p$, then $p\mid k$.
\end{enumerate}
Then if we assume
$$\sum_{d=1}^{\infty}\frac{|\mu(d)|}{l(d)}<\infty,$$
we have
$$\textbf{d}(\A_{E,k})=\sum_{c|k}\mu(c)\textbf{d}(\B_{E,ck}).$$
Furthermore, we get
$$\textbf{d}(\A_{E,k})=\sum_{d=1}^{\infty}\frac{\mu(d)}{l(dk)}.$$
\end{Theorem}

\begin{Remark}
\label{rmk1}
We are generalizing the result of Sanna and Tron \cite{ST} on the Fibonacci sequence in the following sense: Consider the algebraic group
$$A:x^{2}-5y^2=1$$
and the point $\alpha=(3/2,1/2)\in A(\Q)$. Cubre and Rouse \cite[Lemma 9]{CR} showed that the $n$-th iteration of $\alpha$ has the form
$$n\alpha=\left(\frac{L_{2n}}{2},\frac{F_{2n}}{2}\right),$$
where $(L_{n})_{n\geq 1}$ represents the Lucas sequence defined inductively by $L_{0}=2$, $L_{1}=1$, and $L_{n}=L_{n-1}+L_{n-2}$ for $n\geq 2$.
Elliptic divisibility sequences may be regarded as naturally attached divisibility sequences on a fixed elliptic curve which is another instance of an algebraic group.   
\end{Remark}      

\section{Preliminaries}

First, we note some basic properties of an elliptic divisibility sequence.

\begin{Lemma}
\label{PreLem}
For a given minimal elliptic divisibility sequence $\D=(D_{n})_{n\geq 1}$ with a nontorsion point $P\in E(\Q)$ and for all positive integers $m,~ n$, we have:
\begin{enumerate}[(a)]
	\item $n \mid D_{m}~ \Leftrightarrow ~r_{n} \mid m~ \Leftrightarrow~[m]P\equiv O~ (mod~n),$
where the last congruence is defined on the N\'eron model of $E/\Q$, specially for primes of bad reduction.
	\item $m\mid n$ $\Rightarrow$ $r_{m}\mid r_{n}$.
	\item $m\mid n$ $\Rightarrow$ $g(m)\mid g(n)$.
	\item $n \mid g(m)$ $\Leftrightarrow$ $l(n)\mid m$.
	\item $r_{\lcm(m,n)}=\lcm(r_{m},r_{n})$.
	\item $l(\lcm(m,n))=\lcm(l(m),l(n))$.\\
\end{enumerate}
Moreover, if we define
$$\A_{E}=\{g(n): n\geq 1\},$$
\begin{enumerate}[(g)]
\item $n \in \A_{E}$ $\Leftrightarrow$ $n=g(l(n))$.
\end{enumerate}
\end{Lemma}

\begin{proof}
(a), (b), (c) follow from the definition. Using (a), we get (d),
\begin{displaymath}
\left.
\begin{tabular}{c} $n\mid g(m)$
\end{tabular} \right.
\Leftrightarrow
\left(
\begin{tabular}{c} $n\mid m$ \\ $n \mid D_{m}$
\end{tabular} \right)
\Leftrightarrow
\left(
\begin{tabular}{c} $n\mid m$ \\ $r_{n}\mid m$
\end{tabular} \right)
\Leftrightarrow
\left.
\begin{tabular}{c} $l(n)\mid m.$
\end{tabular} \right.
\end{displaymath}
From (b), we know
$$\lcm(r_{m},r_{n})\mid r_{\lcm(m,n)}.$$
On the other hand, we also know
$$D_{\lcm(r_{m},r_{n})}\equiv 0~(\text{mod}~ m) ~\text{and}~ D_{\lcm(r_{m},r_{n})}\equiv 0 ~(\text{mod}~ n).$$
Thus,
$$D_{\lcm(r_{m},r_{n})}\equiv 0~(\text{mod} \lcm(m,n))$$
and 
$$r_{\lcm(m,n)}\mid \lcm(r_{m},r_{n})$$
we get (e).
(f) follows from the observation using (e)
\begin{align*}
l(\lcm(m,n)) &= \lcm(\lcm(m,n), r_{\lcm(m,n)})\\
             &= \lcm(\lcm(m,n), \lcm(r_{m},r_{n}))\\
						 &= \lcm(\lcm(m,r_{m}),\lcm(n,r_{n})).
\end{align*}
For (g), since $n$ divides both $l(n)$ and $D_{l(n)}$, we know $n \mid g(l(n))$ for all positive integers $n$. Thus if $k\in \A_{E}$, then $k=g(n)$ for some positive integer $n$. By (d), this implies $g(l(k))\mid g(n)=k$ and $k=g(l(k))$. The converse holds clearly. 
\end{proof}

\begin{Proposition}
\label{thmm1}
For a given positive integer $k$, $\A_{E,k}\neq \emptyset$ {\it if and only if} $k=g(l(k))=\gcd(l(k),D_{l(k)})$.  
\end{Proposition}

\begin{proof}
The proposition follows directly from Lemma \ref{PreLem} $(f)$.
\end{proof}

For a given set $\mathscr{S}$ of positive integers, we define the {\it set of nonmultiples} of $\mathscr{S}$ as
$$\NN(\mathscr{S})=\{n\geq 1 : s\nmid n~\text{for all}~s\in\mathscr{S} \},$$
and its complement the {\it set of multiples} of $\mathscr{S}$ as
$$\MM(\mathscr{S})=\{n\geq 1 : s\mid n~\text{for some}~s\in\mathscr{S} \}.$$
We also need the following lemma.

\begin{Lemma}
\label{lem10}
If $\mathscr{S}$ is a set of positive integers such that 
$$\sum_{s\in\mathscr{S}}\frac{1}{s}<+\infty,$$
then $\NN(\mathscr{S})$ has an asymptotic density. Moreover, if $1 \notin \mathscr{S}$ then $\textbf{d}(\NN(\mathscr{S}))>0$. 
\end{Lemma}

\begin{proof}
See \cite[Lemma 2.3]{ST}.
\end{proof}

For $\gamma>0$, we define a set of primes
$$\QQ_{\gamma}=\{p : r_{p}\leq p^{\gamma}\}.$$
Also, let
$$\QQ_{\gamma}(x)=\{p\le x :r_{p}\leq p^{\gamma}\}.$$
We prove the following fact.

\begin{Lemma}
\label{lem11}
For all $x,\gamma>0$, we have $\#\QQ_{\gamma}(x)\ll x^{3\gamma}$.
\end{Lemma}

\begin{proof}
From the definition of $\QQ_{\gamma}(x)$, we know
$$e^{\#\QQ_{\gamma}(x)}\leq 2\cdot\prod_{p\in\QQ_{\gamma}(x)}p.$$
Thus, we have an inequality
\begin{equation}
\#\QQ_{\gamma}(x)\leq \log 2 +\sum_{p \in \QQ_{\gamma}(x)}\log p.
\end{equation}
Moreover, for each prime $p \in \QQ_{\gamma}(x)$, $r_{p}\leq p^{\gamma}\leq x^{\gamma}$ which means
$$p\mid D_{r_{p}}\text{, where}~r_{p}\leq x^{\gamma}.$$
Thus, we get the following divisibility relation
$$\prod_{p\in\QQ_{\gamma}(x)}p~\big| \prod_{n\leq x^{\gamma}}D_{n},$$
and we know from (1),
\begin{equation}
\label{eq2}
\#\QQ_{\gamma}(x)\leq \log 2 +\sum_{n\leq x^{\gamma}}\log D_{n}.
\end{equation}
When we denote $\hat{h}_{E}(P)>0$ the canonical height of $P$, we have the limit essentially proven by Siegel \cite[Remark 11]{SS}
$$\lim_{n\to\infty}\frac{\log|D_{n}|}{n^2}=\hat{h}_{E}(P),$$
and using standard Bachmann-Landau notation
$$\log D_{n} =\mathcal{O}(n^{2}),$$
where the hidden constant depends on $\hat{h}_{E}(P)$.
Therefore, \eqref{eq2} tells
$$\#\QQ_{\gamma}(x) \ll \sum_{n\leq x^{\gamma}}n^{2}\leq \frac{x^{\gamma}(x^{\gamma}+1)(2x^{\gamma}+1)}{6},$$
which implies the desired result.
\end{proof}

\begin{Remark}
One can compare Lemma \ref{lem11} to the analogous lemma for the Fibonacci sequence \cite[Lemma 2.4]{ST}. Let $(F_{n})_{n\geq 1}$ be the Fibonacci sequence and let $r_{F,n}$ be the {\it rank of apparition} of $n$ for the Fibonacci sequence, i.e.,
$$r_{F,n}=\min\{r\geq 1 : n\mid F_{r}\}.$$
Also, define
$$\QQ_{F,\gamma}(x)=\{p : r_{F,p}\leq p^{\gamma}\}.$$
Then for all $x,\gamma >0$, we have $$\#\QQ_{F,\gamma}(x)\ll x^{2\gamma}.$$ 
The difference between the Fibonacci case and the EDS case is the result of the different growth rates of the Fibonacci sequence and the EDS. Note that $$F_{n}\leq 2^{n},~~~\text{for all positive integers $n$,}$$
whereas the EDS satisfies the above limit by Siegel.
\end{Remark}

\begin{Lemma}
\label{lem1}
Let $\D=(D_{n})_{n\geq 1}$ be a minimal EDS, let $n\geq 1$, and let $p$ be a prime satisfying $p\mid D_{n}$. For all $m\geq 1$ we have
$$\nu_{p}(D_{m n})\geq \nu_{p}(m D_{n}).$$

\end{Lemma}

\begin{proof}
\cite[Lemma 5]{SS}
\end{proof}

\begin{Lemma}
\label{lem2}
Let $\D$ be a minimal EDS associated to an elliptic curve $E/\Q$ and point $P \in E(\Q)$ and let $p$ be a prime. Then
$$r_{n}\mid \#\mathfrak{C}(\Z/n\Z),$$
where $\mathfrak{C}$ is the N\'eron model of E, that is to say, a group scheme over $\Spec(\mathbb{Z})$ whose generic fiber is $E/\Q$. One can find a detailed explanation of the N\'eron model in \cite[IV]{SIL} 
Moreover, if $p$ is a prime with $P\in E_{ns}(\F_{p})$, then
$$r_{p}\leq (\sqrt{p}+1)^2.$$
If $P \in E_{ns}(\F_{p})$ and $E$ has a bad reduction at $p$, then $r_{p}$ divides $p-1$, $p+1$, or $p$ depending respectively on whether the reduction is split multiplicative, non-split multiplicative, or additive.  
\end{Lemma}

\begin{proof}
\cite[Lemma 6]{SS}
\end{proof}

Also, we can show a nice structural property of $\A_{E,k}$.
\begin{Theorem}
\label{thm1}
For each positive integer $k$ with $\A_{E,k}\neq \emptyset$, denote
$$\LL_{k}=\left\{ p : p|k \right\}\cup \left\{ \frac{l(kp)}{l(k)}:p\nmid k\right\}.$$
Then
$$\A_{E,k}=\{l(k)m : m\in \NN(\LL_{k})\}.$$ 

\end{Theorem}

\begin{proof}
If $n \in \A_{E,k}$, then $k\mid D_{n}$ and $l(k)\mid n$ by (d) of Lemma \ref{PreLem}. Thus it is sufficient to prove that $l(k)m\in \A_{E,k}$ for some $m$ if and only if $m \in \NN(\LL_{k})$.\\
\indent
First, note that $l(k)m \in \A_{E,k}$ if and only if the valuation at $p$
\begin{equation}
\label{eqq3}
\nu_{p}(\gcd(l(k)m, D_{l(k)m}))=\nu_{p}(k)
\end{equation}
for all prime $p$.
Assume the case when $p$ divides $k$. For all positive integer $m$, we have $r_{p}\mid l(k)m$. Write $l(k)m=r_{p}m'$. From Lemma \ref{lem1}, we have
\begin{equation}
\label{eq4}
\nu_{p}(D_{l(k)m})=\nu_{p}(D_{r_{p}m'})\geq \nu_{p}(m'D_{r_{p}})=\nu_{p}(m')+\nu_{p}(D_{r_{p}}).
\end{equation}
By the way, for all prime $p$, Lemma \ref{lem2} tells us
$$r_{p}\leq (\sqrt{p}+1)^{2}<2p.$$
Thus $\nu_{p}(r_{p})\leq 1$. We also know $\nu_{p}(D_{r_{p}})\geq 1$ from the definition. Therefore $\nu_{p}(D_{r_{p}})\geq \nu_{p}(r_{p})$ and we get from (\ref{eq4})
$$\nu_{p}(D_{l(k)m})\geq \nu_{p}(m')+\nu_{p}(r_{p})=\nu_{p}(m'r_{p})=\nu_{p}(l(k)m)=\nu_{p}(l(k))+\nu_{p}(m).$$
Thus we get
\begin{equation}
\label{eq6}
\nu_{p}(\gcd(l(k)m, D_{l(k)m}))=\nu_{p}(l(k)m)=\nu_{p}(l(k))+\nu_{p}(m).
\end{equation}
On the other hand, from Theorem \ref{thmm1}, $\A_{E,k}\neq \emptyset$ implies $k=\gcd(l(k),D_{l(k)})$. With $m=1$, we get
$$\nu_{p}(k)=\nu_{p}(\gcd(l(k),D_{l(k)}))=\nu_{p}(l(k))$$
and with (\ref{eq6})
$$\nu_{p}(\gcd(l(k)m,D_{l(k)m}))=\nu_{p}(k)+\nu_{p}(m).$$
Therefore, the equivalence in (\ref{eqq3}) is true if and only if $p \nmid m $.\\
\indent
We consider now the case when $p$ does not divide $k$. Then (\ref{eqq3}) holds if and only if 
$$p \nmid \gcd(l(k)m,D_{l(k)m}),$$
which is equivalent to $l(p)\nmid l(k)m$ and again equivalent to
$$\frac{l(kp)}{l(k)}=\frac{\lcm(l(k),l(p))}{l(k)}\nmid m,$$
which completes the proof of Theorem \ref{thm1}. 

\end{proof}

\section{The asymptotic density of $\A_{E,k}$ for finitely anomalous elliptic curve}

From Remark \ref{rmk1}, the Fibonacci sequence arises from a simpler algebraic group than the EDS which is determined by an elliptic curve with a fixed point. Thus, one can expect more complicated sequences. For instance, it is much simpler for the Fibonacci sequence to control the rank of apparition $r_{p}$. For instance, it satisfies the following property \cite[Lemma 2.1 (iii)]{ST}:
$$r_{p} \mid p-\big(\frac{p}{5}\big),~\text{where $\big(\frac{p}{5}\big)$ is a Legendre Symbol},$$
whereas for the distribution of $r_{p}$, which is essentially dividing $\#E(\mathbb{F}_{p})$, the order of $E/\Q$ modulo $p$, is involved. Especially, the distribution of $\#E(\F_{p})$ is related to many conjectures such as the Sato-Tate conjecture \cite{TATE} and the Lang-Trotter conjecture \cite{LT}. Although, we can still prove some asymptotic results by restricting our consideration to a certain nice family of elliptic curves.

\begin{Definition}
Given an elliptic curve $E/\Q$, we call $p$ an {\it anomalous prime} of $E$ if it satisfies
$$\#E(\F_{p})=p.$$
That is to say, its trace of Frobenius $a_{p}(E)=p+1-\#E(\F_{p})$ is $1$. We say $E/\Q$ is {\it finitely anomalous} if there are only finitely many places $p$ with $\#E(\F_{p})=p$.
\end{Definition}

We have an observation by Silverman and Stange \cite{SS}.
\begin{Remark}
Let $E/\Q$ be an elliptic curve and $\textbf{D}$ be an EDS associated to the curve $E(\Q)$. An aliquot cycle of length one, i.e., a prime $p\geq 7$ satisfying $r_{p}(\textbf{D})=p$ implies, using Hasse's estimate,
$$r_{p}(\textbf{D})=p~ \Leftrightarrow ~\#E(\F_{p})=p.$$  
\end{Remark}

\begin{Remark}
All elliptic curves with nontrivial torsion group over $\Q$ are always finitely anomalous. This could be observed by the natural embedding of the torsion group. For each good prime $p>7$, 
$$E(\Q)_{tors}\hookrightarrow E(\mathbb{F}_{p}).$$ On the other hand, elliptic curves with nontrivial torsion group over $\Q$ are rather sparsely existing. For instance, one can refer \cite[Theorem 1.1]{GT}.
Independently, Ridgdill \cite[1.3]{RI} proved that any elliptic curve which has $2$-torsion is always finitely anomalous. The proof uses the fact that $a_{p}(E)$ is always even if $p$ does not divide the conductor of $E/\Q$ for $p$ greater than 2. He also notes some classifications of finitely anomalous elliptic curves under some special cases of Galois representation.  
\end{Remark}

\begin{Theorem}
\label{thm2}
Let $E/\Q$ be a finitely anomalous elliptic curve and let $k>0$ be an integer. Then the following holds.
\begin{enumerate}[(a)]
\item The asymptotic density $\textbf{d}(\A_{E,k})$ exists.
\item $\textbf{d}(\A_{E,k})>0$ if and only if $\A_{E,k}\neq \emptyset$.
\end{enumerate}
\end{Theorem}

\begin{proof}
If $\A_{E,k}=\emptyset$ then $\textbf{d}(\A_{E,k})=0$. Assume $\A_{E,k}\neq \emptyset$. We can induce $k=\gcd(l(k),D_{l(k)})$ by Theorem \ref{thmm1}. Then we have
\begin{equation}
\label{eqq6}
\sum_{n\in\LL_{k}}\frac{1}{n}\ll \sum_{p}\frac{1}{l(kp)}\leq \sum_{p}\frac{1}{l(p)}.
\end{equation}
Since $E$ is a finitely anomalous elliptic curve, for all but finitely many primes
$$ \#E(\F_{p})\neq p~ \text{and}~ \gcd(p,r_{p})=1. $$
Thus, for all but finitely many primes, the following holds
$$l(p)=\lcm(p,r_{p})=p\cdot r_{p}.$$
Therefore, from (\ref{eqq6}), we get

\begin{equation}
\label{eq7}
\sum_{n\in\LL_{k}}\frac{1}{n}\ll \sum_{p}\frac{1}{l(kp)}\leq \sum_{p}\frac{1}{l(p)}\ll \sum_{p}\frac{1}{p\cdot r_{p}}.
\end{equation}
For any $\gamma \in (0,\frac{1}{3})$, we have
$$\sum_{p \notin \QQ_{\gamma}}\frac{1}{p\cdot r_{p}}<\sum_{p\notin\QQ_{\gamma}}\frac{1}{p^{1+\gamma}}<\sum_{n}\frac{1}{n^{1+\gamma}}<+\infty.$$
Furthermore, by Lemma \ref{lem11} and partial summation,
$$\sum_{p\in\QQ_{\gamma}}\frac{1}{p\cdot r_{p}}<\sum_{p\in\QQ_{\gamma}}\frac{1}{p}=\frac{\#\QQ_{\gamma}(t)}{t}\Big|_{t=2}^{+\infty} + \int_{2}^{+\infty} \frac{\#\QQ_{\gamma}(t)}{t^2}~dt \ll \int^{+\infty}_{2}\frac{dt}{t^{2-3\gamma}}<+\infty.$$
We get 
$$\sum_{n \in \LL_{k}}\frac{1}{n}<+\infty.$$
Obviously, $1\notin \LL_{k}$ and thus by Lemma \ref{lem10}, $\textbf{d}(\NN(\LL_{k}))>0$. Therefore by Theorem \ref{thm1}, the asymptotic density of $\A_{E,k}$ has positive density.
\end{proof}

We can actually prove the above result without the finitely anomalous condition using a result of Serre \cite{SER} provided $E$ does not have complex multiplication.

\begin{Theorem}
Let $E/\Q$ be an elliptic curve without complex multiplication and let $k>0$ be an integer. The following holds.
\begin{enumerate}[(a)]
\item The asymptotic density $\textbf{d}(\A_{E,k})$ exists.
\item $\textbf{d}(\A_{E,k})>0$ if and only if $\A_{E,k}\neq \emptyset$.
\end{enumerate}
Moreover, assuming the generalized Riemann hypothesis (GRH), one can prove the same result for elliptic curves with complex multiplication. 
\end{Theorem}

\begin{proof}
From Theorem (\ref{eqq6}), we know
\begin{equation}
\label{eqq7}
\sum_{n\in\LL_{k}}\frac{1}{n}\ll \sum_{p}\frac{1}{l(kp)}\leq \sum_{p}\frac{1}{l(p)}~=\sum_{\substack{p \\ \text{non-anomalous}}}\frac{1}{p\cdot r_{p}}+\sum_{\substack{p \\ \text{anomalous}}}\frac{1}{p},
\end{equation}
where the summation of non-anomalous primes can be treated in the manner of Theorem \ref{thm2}. Therefore, it is reduced to proving the finiteness of the sum
$$\sum_{\substack{p\\ \text{anomalous}}}\frac{1}{p}.$$
We use the following result of Serre \cite[Corollary 1, p. 174]{SER}.
\begin{equation}
\label{SerEq}
\#\{p\leq x : p\nmid N,~ \text{$p$ is anomalous}\}\ll \frac{x}{(\log x)^{5/4-\delta}},~~\text{for all}~\delta>0,  
\end{equation}
where $N$ is the conductor of $E$. Denote
\begin{equation}
\label{eqq8}
A(x)=\sum_{\substack{1\leq p\leq x\\ \text{anomalous}}}1.
\end{equation}
Using the above notation and Abel's summation formula, for a fixed number $x, y$ and any $0<\delta<\frac{1}{4}$, we can write the sum
\begin{align*}
\sum_{\substack{y\leq p \leq x\\\text{anomalous}}}\frac{1}{p} 
& =A(x)\cdot\frac{1}{x}-A(y)\cdot\frac{1}{y}+\int^{x}_{y}A(t)\cdot\frac{1}{t^{2}}~dt\\
& \ll \frac{x}{(\log x)^{5/4-\delta}}\cdot\frac{1}{x}+\int^{x}_{y}\frac{t}{(\log t)^{5/4-\delta}}\cdot\frac{1}{t^{2}}~dt\\
& \ll \frac{1}{(\log x)^{5/4-\delta}}+\int^{x}_{y}\frac{1}{t(\log t)^{5/4-\delta}}~dt\\
& \ll \frac{1}{(\log x)^{5/4-\delta}}+\int^{\log x}_{\log y}\frac{1}{s^{5/4-\delta}}~ds\\
& = \frac{1}{(\log x)^{5/4-\delta}}-\frac{1}{1/4-\delta}\Big[\frac{1}{(\log x)^{1/4-\delta}}-\frac{1}{(\log y)^{1/4-\delta}}\Big].
\end{align*}

Thus, as $x \to \infty$, 
$$\sum_{\substack{y\leq p \\ \text{anomalous}}}\frac{1}{p}$$
is finite, which completes the proof for elliptic curves without complex multiplication. For proving the same result for elliptic curves having complex multiplication, we use the result of M. R. Murty, V. K. Murty, and N. Saradha \cite[Remark 4.3 (iii)]{MMS} which implies the following inequality. Assuming the $GRH$,
$$\#\{p\leq x : p \nmid N,~ \text{$p$ is anomalous}\}\ll x^{1/2}(\log x)^{2},$$
where $N$ is again the conductor of $E$. Using the same notation above, for a fixed number $x, y$,
\begin{align*}
\sum_{\substack{y\leq p \leq x\\\text{anomalous}}}\frac{1}{p} 
& =A(x)\cdot\frac{1}{x}-A(y)\cdot\frac{1}{y}+\int^{x}_{y}A(t)\cdot\frac{1}{t^{2}}~dt\\
& \ll x^{1/2}(\log x)^{2}\cdot \frac{1}{x}+\int^{x}_{y}t^{1/2}(\log t)^{2}\cdot\frac{1}{t^{2}}~dt\\
& \ll \frac{(\log x)^{2}}{x^{1/2}}+\int^{x}_{y}\frac{(\log t)^{2}}{t^{3/2}}~dt\\
& = \frac{(\log x)^{2}}{x^{1/2}}-2\left[\frac{(\log x)^{2}+4\log x+8}{\sqrt{x}}-\frac{(\log y)^{2}+4\log y+8}{\sqrt{y}}\right].\\
\end{align*}
Thus, as $x \to \infty$,
$$\sum_{\substack{y\leq p \\ \text{anomalous}}}\frac{1}{p}$$
is finite, which shows the result holds for elliptic curves with complex multiplication under the $GRH$ holds.
\end{proof}

\begin{Remark}
Assuming the $GRH$, M. Murty, V. Murty, and Saradha \cite{MMS} get the improved the result of Serre \cite{SER}, which is
\begin{displaymath}
\#\{p\leq x:p\nmid N, a_{p}(E)=a\} 
\ll \left\{ \begin{array}{ll} 
x^{4/5}(\log x)^{-1/5} & \textrm{if $a\neq 0$,}\\ 
x^{3/4} & \textrm{if $a=0$.}\\
\end{array} \right.
\end{displaymath}
On the other hand, when $a\neq 0$, the Lang-Trotter conjecture \cite{LT} suggests
$$\#\{p\leq x:p\nmid N, a_{p}(E)=a\} \sim c\cdot \frac{\sqrt{x}}{\log x}, $$
for some constant $c>0$ depending on $E$. 

\end{Remark}

Furthermore, for a fixed elliptic curve $E/\Q$ and a positive integer $k$,  define $\B_{E,k}$ be the set of positive integers $n$ satisfying the following conditions.
\begin{enumerate}
	\item $k \mid g(n).$
	\item If $p\mid g(n)$ for some prime $p$, then $p\mid k$.
\end{enumerate}

Then we can compute the asymptotic density of $\B_{E,k}$.

\begin{Lemma}
\label{lem3}
Let $E/\Q$ be an elliptic curve and $k$ be a given positive integer. Then we can get
$$\#\B_{E,k}(x)=\sum_{\substack{d\leq x\\ (d,k)=1}}\mu(d)\Big\lfloor \frac{x}{l(dk)}\Big\rfloor.$$
Furthermore, if we assume
$$\sum_{d=1}^{\infty}\frac{|\mu(d)|}{l(d)}<\infty,$$
then the asymptotic density of $\B_{E,k}$ exists and is absolutely convergent to
$$\textbf{d}(\B_{E,k})=\sum_{(d,k)=1}\frac{\mu(d)}{l(dk)}.$$
\end{Lemma}

\begin{proof}
The proof essentially follows \cite[Lemma 4.4]{ST}. For given positive integers $n$ and $d$, we define a function
$$\varrho(n,d)=\left\{ \begin{array}{ll} 
1 & \textrm{if $d\mid D_{n},$}\\ 
0 & \textrm{if $d\nmid D_{n}.$} \end{array} \right.$$
For all positive integers $n$, if two given integers $d$ and $e$ are relatively prime, we can easily observe
$$\varrho(n,de)=\varrho(n,d)\varrho(n,e).$$
That is to say, $\varrho$ is multiplicative in the second coordinate. Note that $n\in \B_{E,k}$ if and only if $l(k)\mid n$ and $\varrho(n,p)=0$ for all prime $p$ satisfying $p\mid n$ but $p\nmid k$. Thus, we can count

\begin{align}
\label{eq8}
\#\B_{E,k}(x) & =\sum_{\substack{n\leq x\\l(k)\mid n}}\prod_{\substack{p\mid n\\p\nmid k}}(1-\varrho(n,p)),\\
\label{eq9}
 &=\sum_{\substack{n\leq x\\l(k)\mid n}}\sum_{\substack{d\mid n\\(d,k)=1}}\mu(d)\varrho(n,d)=\sum_{\substack{d\leq x\\ (d,k)=1}}\mu(d)\sum_{\substack{m\leq x/d\\l(k)\mid dm}}\varrho(dm,d).
\end{align}

Furthermore, we can note that 
$$\varrho(dm,d)=1 ~\text{and}~l(k)\mid dm  \Leftrightarrow \lcm(z(d),l(k))\mid dm, $$
which is equivalent to, if $(d,k)=1$, the following divisibility of $m$:
$$\frac{\lcm(d,\lcm(z(d),l(k)))}{d}=\frac{\lcm(l(d),l(k))}{d}=\frac{l(dk)}{d},$$
and we know
$$\frac{l(dk)}{d}~\Big|~ m.$$
Thus, we can write the last sum of (\ref{eq9}) as
$$\sum_{\substack{m\leq x/d\\l(k)|dm}}\varrho(dm,d)=\sum_{\substack{m\leq x/d\\l(dk)/d|m}}1=\Big\lfloor \frac{x}{l(dk)}\Big\rfloor.$$
We can deduce from (\ref{eq8}),

\begin{equation}
\label{eq10}
\#\B_{E,k}(x)=\sum_{\substack{d\leq x\\ (d,k)=1}}\mu(d)\Big\lfloor \frac{x}{l(dk)}\Big\rfloor.
\end{equation}

Moreover, if we have the additional assumption
$$\sum_{d=1}^{\infty}\frac{|\mu(d)|}{l(d)}<\infty,$$
then from (\ref{eq10}),

\begin{equation}
\label{eq11}
\#\B_{E,k}(x)=\sum_{\substack{d\leq x\\ (d,k)=1}}\mu(d)\Big\lfloor \frac{x}{l(dk)}\Big\rfloor =x\sum_{\substack{d\leq x\\(d,k)=1}}\frac{\mu(d)}{l(dk)}-\sum_{\substack{d\leq x\\(d,k)=1}}\mu(d)\Big\{\frac{x}{l(dk)}\Big\},
\end{equation}
where $\{\cdot\}$ represents the sawtooth function, i.e., $\{x\}=x-\lfloor x \rfloor $. From the assumption,
$$
\sum_{\substack{d\leq x\\ (d,k)=1}}\frac{|\mu(d)|}{l(dk)}\leq \sum^{\infty}_{d=1}\frac{|\mu(d)|}{l(d)}<+\infty.
$$
Moreover, we can bound
\begin{align*}
\Big|\sum_{\substack{d\leq x\\(d,k)=1}}\mu(d)\Big\{\frac{x}{l(dk)}\Big\}\Big| 
& \leq \sum_{d\leq x}|\mu(d)|\Big\{\frac{x}{l(dk)}\Big\} \\
& =\mathcal{O}(x^{1/2})+\sum_{x^{1/2}\leq d\leq x}|\mu(d)|\Big\{\frac{x}{l(dk)}\Big\} \\
& \leq \mathcal{O}(x^{1/2})+x\sum_{d\geq x^{1/2}}\frac{|\mu(d)|}{l(d)}=o(x),
\end{align*}
which could be seen from the assumption: as the convergence implies
$$\sum_{d\geq x^{1/2}}\frac{|\mu(d)|}{l(d)}\to 0 ~~\text{as}~~x\to +\infty.$$
Thus, from (\ref{eq11}), we have
$$\frac{\#\B_{E,k}(x)}{x}\to \sum_{(d,k)=1}\frac{\mu(d)}{l(dk)}~~\text{as}~~x\to\infty.$$
\end{proof}

\begin{Theorem}
\label{thm3}
Let $E/\Q$ be an elliptic curve which satisfies one of the following:
\begin{enumerate}
\item $E/\Q$ does not have complex multiplication.
\item $E/\Q$ has complex multiplication, and either assume the generalized Riemann Hypothesis or that being $E/\Q$ is finitely anomalous.
\end{enumerate}
Let $k$ be a fixed positive integer. If we assume
\begin{equation}
\label{maincondition}
\sum_{d=1}^{\infty}\frac{|\mu(d)|}{l(d)}<\infty,
\end{equation}
we have
$$\textbf{d}(\A_{E,k})=\sum_{c|k}\mu(c)\textbf{d}(\B_{E,ck}).$$
Furthermore,
we get
$$\textbf{d}(\A_{E,k})=\sum_{d=1}^{\infty}\frac{\mu(d)}{l(dk)}.$$
\end{Theorem}

\begin{proof}
By the definition of $\B_{E,k}$, we can observe that
$$\#\A_{E,k}(x)=\sum_{c|k}\mu(c)\#\B_{E,ck}(x),$$
by the inclusion-exclusion principle. Thus, by Lemma \ref{lem3}, we have
$$\textbf{d}(\A_{E,k})=\sum_{c|k}\mu(c)\textbf{d}(\B_{E,ck}).$$
Moreover, if the following holds
$$\sum_{d=1}^{\infty}\frac{|\mu(d)|}{l(d)}<\infty,$$
Lemma \ref{lem3} tells
\begin{align*}
\textbf{d}(\A_{E,k}) &=\sum_{c|k}\mu(c)\textbf{d}(\B_{E,ck}).
=\sum_{c|k}\mu(c)\sum_{(e,ck)=1}\frac{\mu(e)}{l(eck)}\\
&=\sum_{c|k}\sum_{e,k}\frac{\mu(ce)}{l(eck)}=\sum^{\infty}_{f=1}\frac{\mu(f)}{l(fk)},
\end{align*}
where the absolute convergence in Lemma \ref{lem3} guarantees the rearrangement of the sum.
\end{proof}

\section{Examples}

\begin{Remark}
The additional assumption in Lemma \ref{lem3} and Theorem \ref{thm3}
\begin{equation}
\label{sum}
\sum_{d=1}^{\infty}\frac{|\mu(d)|}{l(d)}<\infty
\end{equation}
could be compared to the inverse of the Riemann zeta function, i.e., if $s$ is a complex number with $\Re(s)>1$,
$$\sum_{d=1}^{\infty}\frac{\mu(d)}{d^{s}}=\frac{1}{\zeta(s)}.$$
Whereas the case $s=1$ looks closest to our case, in which case $\zeta(s)$ has a simple pole and thus converges to $0$. Also, note that
$$\sum_{d=1}^{\infty}\frac{|\mu(d)|}{l(d)}<\sum^{\infty}_{d=1}\frac{|\mu(d)|}{d}$$
and the second sum diverges. Thus we can conclude the convergence of (\ref{sum}) is highly dependent on the distribution of $r_{d}$, i.e., the group structure of each fiber of the N\'eron model of $E/\Q$ . However, (\ref{sum}) increases extremely slowly in the following examples.
\end{Remark}

\begin{Example}
Let $E_{1}$ be the elliptic curve defined by the Weierstrass equation
$$E_{1} : y^{2}+y=x^{3}-x,~~~~P=(0,0).$$
Then the EDS associated to $(E_{1},P)$ is the sequence
$$1, 1, 1, 1, 2, 1, 3, 5, 7, 4, 23, 29, 59, 129, 314, 65, 1529, 3689, 8209, 16264, \cdots.$$
Denote
$$\BB_{E_{1}}(n)=\sum_{d=1}^{n}\frac{|\mu(d)|}{l(d)}.$$

Note that $\BB_{E_{1}}(n)$ converges as $n\to \infty$ implies \ref{sum}. Also, it is worth noting that $\BB_{E_{1}}(n)$ is an increasing sequence. We can see how slowly it increases from table \ref{table1}.
\begin{table}[!ht]
\centering
\caption{How slowly $\BB_{E_{1}}(n)$ grows}
\label{table1}
\begin{tabular}{|c|c|}
\hline
$n$ & $\mathcal{B}_{E_{1}}(n)$ \\ \hline
50  & 1.27363664516258 \\ \hline
100 & 1.30220546776075 \\ \hline
150 & 1.31814339876107 \\ \hline
200 & 1.32279002218373 \\ \hline
250 & 1.32537586326977 \\ \hline
300 & 1.32806568329757 \\ \hline
350 & 1.32934443230431 \\ \hline
400 & 1.33105981658652 \\ \hline
\end{tabular}
\end{table}
\end{Example}

Also, denote
$$\BB_{E_{1}}(k,n)=\sum_{d=1}^{n}\frac{\mu(d)}{l(kd)},$$
which represents the approximation of $\textbf{d}(\A_{E_{1},k})$ as $n \to \infty$ acccording to the Theorem \ref{thm3} when (\ref{sum}) is satisfied.\\

\begin{table}[!ht]
\centering
\caption{How $\mathcal{B}_{E_{1}}(k,n)$ differs for $k=1,2,5$}
\label{table4}
\begin{tabular}{|c|c|c|c|}
\hline
$n$ & $\mathcal{B}_{E_{1}}(1,n)$ & $\mathcal{B}_{E_{1}}(2,n)$ & $\mathcal{B}_{E_{1}}(5,n)$ \\ \hline
50  & 0.835303029452152          & 0.0219930355845770         & 0.00424286547549155        \\ \hline
100 & 0.818084769942769          & 0.0225599636689219         & 0.00455796737986152        \\ \hline
\end{tabular}
\end{table}

The next instance is an EDS which increases very rapidly. Silverman and Stange \cite{SS} also deal with the following EDS for different purpose.

\begin{Example}
Let $E_{2}$ be the elliptic curve defined by the Weierstrass equation
$$E_{2}:y^2 + y = x^3 + x^2 - 1291874622406186x+ 17872226251073822113702$$with
$$P=(20751503,1073344).$$
Then the EDS associated to $(E_{2},P)$ is the sequence
$$1, 2146689, 286883381041833542301, 60768120452650698495048133538894517,$$
$$23611096745951856413517153888476821489410524330413499766653328,\cdots$$

We can see the EDS increases very rapidly. Also, denote
$$\BB_{E_{2}}(n)=\sum_{d=1}^{n}\frac{|\mu(d)|}{l(d)}$$
and we can see how it grows from table \ref{table2}.

\begin{table}[!h]
\centering
\caption{How slowly $\BB_{E_{2}}(n)$ grows}
\label{table2}
\begin{tabular}{|c|c|}
\hline
$n$ & $\mathcal{B}_{E_{2}}(n)$ \\ \hline
50  & 1.44883391429462 \\ \hline
100 & 1.48730064005378 \\ \hline
150 & 1.50096312029532 \\ \hline
200 & 1.51957559235974 \\ \hline
250 & 1.52472347568884 \\ \hline
300 & 1.53317425352626 \\ \hline
350 & 1.53563342357803 \\ \hline
400 & 1.53866052239358 \\ \hline
\end{tabular}
\end{table}
\end{Example}
We can again define
$$\BB_{E_{2}}(k,n)=\sum_{d=1}^{n}\frac{\mu(d)}{l(kd)},$$
which represents the approximation of $\textbf{d}(\A_{E_{2},k})$ as $n \to \infty$ acccording to the Theorem \ref{thm3} when (\ref{sum}) is satisfied.

\begin{table}[!ht]
\centering
\caption{How $\mathcal{B}_{E_{2}}(k,n)$ differs for $k=1,3,8$}
\label{table3}
\begin{tabular}{|c|c|c|c|}
\hline
$n$ & $\mathcal{B}_{E_{1}}(1,n)$ & $\mathcal{B}_{E_{1}}(3,n)$ & $\mathcal{B}_{E_{1}}(8,n)$ \\ \hline
50  & 0.700013578679941          & 0.0585355444055008         & -0.0199178419306465        \\ \hline
100 & 0.717097840727588          & 0.0638815953096523         & -0.0189660125468538        \\ \hline
\end{tabular}
\end{table}

\begin{Remark}
Of course, we cannot say that Table \ref{table4} and Table \ref{table3} give the actual estimation of $\textbf{d}(\A_{E_{1},k})$ and $\textbf{d}(\A_{E_{2},k})$ without the assumption (\ref{sum}). If (\ref{sum}) holds, we can expect $\mathcal{B}_{E_{1}}(8,n)$ to converge to a small positive number as $n \to \infty$.
\end{Remark}

\section*{Acknowledgement}
The author would like to thank Joseph H. Silverman for
his helpful advice.


\nocite{CR}
\nocite{GT}
\nocite{LT}
\nocite{MMS}
\nocite{RI}
\nocite{SER}
\nocite{SIL}
\nocite{SIL1}
\nocite{SS}
\nocite{ST}
\nocite{TATE}


\begin{thebibliography}{11}
\bibitem{CR} P. Cubre, J. Rouse, Divisibility properties of the Fibonacci entry point, {\it Proc. Amer. Math. Soc}. {\bf 142}, (2014), no. 11, 3771-3785.
\bibitem{GT} E. Gonz\'alez-Jim\'enez, J. Tornero, On the ubiquity of trivial torsion on elliptic curves, {\it Arch. Math. (Basel)}. {\bf 95} (2010), no. 2, 135-141.
\bibitem{LT} S. Lang, H. Trotter, Frobenius distributions in $GL_{2}$ extensions, {\it Lecture Notes in Mathematics}, {\bf 504} (1976), Springer-Verlag, Heidelberg.
\bibitem{MMS} M. Ram Murty, V. Kumar Murty, and N. Saradha, Modular forms and the Chebotarev density theorem, {\it Amer. J. Math.} {\bf 110} (1988), no. 2, 253-281.
\bibitem{RI} P. C. Ridgdill, On the frequency of finitely anomalous elliptic curves, ProQuest LLC, Ann Arbor, MI, 2010. Thesis (Ph. D.)-University of Massachusetts Amherst.
\bibitem{SER} J. P. Serre, Quelques applications du th\'eor\`eme de densit\'e de Chebotarev, {\it Inst. Hautes \'Etudes Sci. Publ. Math.} {\bf 54} (1981), 323-401.
\bibitem{SIL} J. H. Silverman, Advanced Topics in the Arithmetic of Elliptic Curves, Springer, 1995.
\bibitem{SIL1} J. H. Silverman, The Arithmetic of Elliptic Curves, Springer, 2009.
\bibitem{SS} J. H. Silverman, K. E. Stange, Terms in elliptic divisibility sequences divisible by their indices, {\it Acta Arith.} {\bf 146}, (2011), no. 4, 355-378.
\bibitem{ST} C. Sanna, E. Tron, The density of numbers $n$ having a prescribed G.C.D. with the $n$th Fibonacci number, https://arxiv.org/abs/1705.01805.
\bibitem{TATE} J. Tate, Algebraic cycles and the pole of zeta functions, {\it Arithmetical Algebraic Geometry}, Harper and Row, New York, 1965, 93-110.

 

\end{thebibliography}
\end{document}